\documentclass[reqno,12pt]{amsart}
\usepackage{vmargin}
\setpapersize{A4}

\usepackage{amsmath}

\usepackage{amsfonts}

\usepackage{amssymb}

\usepackage{eufrak}

\usepackage{amscd}

\usepackage{amsthm}

\usepackage{epsfig}

\usepackage{amstext}

\usepackage[all,line,dvips]{xy}
\CompileMatrices 

   \newcommand{\Hom}{\operatorname{Hom}}

\newcommand{\Ad}{\operatorname{Ad}}

\newcommand{\id}{\operatorname{id}}
\newcommand{\Aut}{\operatorname{Aut}}
\newcommand{\eq}[1]{\begin{equation}#1\end{equation}}

 \newcommand{\Ext}{\operatorname{Ext}}
\newcommand{\diag}{\operatorname{diag}}

 \newcommand{\ev}{\operatorname{ev}}

   \theoremstyle{plain}
   \newtheorem{thm}{Theorem}[section]
   \newtheorem{prop}[thm]{Proposition}
   \newtheorem{lemma}[thm]{Lemma}
   \newtheorem{cor}[thm]{Corollary}
   \theoremstyle{definition}
   
   \newtheorem{defn}[thm]{Definition}
   \newtheorem{example}[thm]{Example}
   \theoremstyle{remark}
   
   \newtheorem{remark}[thm]{Remark}

\usepackage{graphicx}

   \numberwithin{equation}{section}

        \date{\today}

\title[Extensions of $C^*$-algebras]{Semi-invertible extensions of $C^*$-algebras}
\author{Vladimir Manuilov and Klaus Thomsen}

\date{\today}

\email{matkt@imf.au.dk}

  \address{Dept. of Mech. and Math.\\
Moscow State University\\
Moscow, 119899, Russia}
\address{Institut for matematiske fag, Ny Munkegade, 8000 Aarhus C,
  Denmark}

\begin{document}

\maketitle

\section{Introduction and statements of results}

The number of examples of $C^*$-algebras for which the semi-group
of extensions by the compact operators is not a group was only
slowly increasing during the first decades following the first
example of J. Anderson,\cite{A}, but recently the pace has picked
up, cf. \cite{HT}, \cite{HS}, \cite{HLSW} and \cite{Se}, and there
are now whole series of $C^*$-algebras $A$ for which it is known
that there are non-invertible extensions of $A$ by the
$C^*$-algebra of compact operators $\mathbb K$. Furthermore, by
considering extensions by general stable $C^*$-algebras the stock
of examples of non-invertible extensions grows considerably.
Indeed, a non-invertible extension of a $C^*$-algebra $A$ by
$\mathbb K$ gives rise to a non-invertible extension of $A$ by $B
\otimes \mathbb K$ for any unital $C^*$-algebra
$B$.\footnote{Tensor the non-invertible extension with $B$ using
the maximal tensor-product, and pull back along the unital
inclusion $A \subseteq A \otimes_{max} B$. It is easy to see that
the resulting extension of $A$ by $B \otimes \mathbb K$ does not
have a completely positive section for the quotient map because
the original extension does not.}

In a different direction the authors have shown that many of the non-invertible
extensions are invertible in a slightly weaker sense, called
\emph{semi-invertibility}. Recall that an extension of a
$C^*$-algebra $A$ by a stable $C^*$-algebra $B$ is invertible when
there is another extension, the inverse, with the property that the
direct sum extension of the two is a split
extension. Semi-invertibility requires only that the sum is
\emph{asymptotically split}, in the sense that there is an asymptotic
homomorphism as defined by Connes and Higson, \cite{CH}, consisting
of right-inverses of the quotient map. It turns out that extensions of
a suspended or a contractible $C^*$-algebra are always
semi-invertible, \cite{MT3}, \cite{MT1}, and in \cite{ST} it was shown that
the
extensions of the reduced group $C^*$-algebra of a free product of
amenable groups are all semi-invertible. The main purpose of the
present paper is to prolonge this list of $C^*$-algebras for which all the
extensions by a separable stable $C^*$-algebra are
semi-invertible.

To explain why semi-invertibility is a natural notion which can be
considered as the best alternative when invertibility fails, we
recall first the central definitions. Let $A$ and $B$ be separable
$C^*$-algebras. The multiplier algebra of $B$ will be denoted by
$M(B)$, the generalized Calkin algebra of $B$ by $Q(B)$ and $q_B :
M(B) \to Q(B)$ is then the canonical surjection. We let
$\Ext(A,B)$ denote the semi-group of unitary equivalence classes
of extensions of $A$ by $B$. Thus elements of $\Ext(A,B)$ are
represented by $*$-homomorphisms $\varphi : A \to Q(B)$ and two
extensions $\varphi, \psi : A \to Q(B)$ are unitarily equivalent
when there is a unitary $u \in M(B)$ such that $\Ad q_B(u) \circ
\varphi = \psi$. The addition $\varphi \oplus \psi$ of two
extensions is defined from a choice of isometries $V_1,V_2 \in
M(B)$ such that $V_1V_1^* + V_2V_2^* = 1$ to be the extension
$$
\left(\varphi \oplus \psi\right)(a)=  q_B\left(V_1\right) \varphi(a)q_B\left(V_1\right)^* +
  q_B\left(V_2\right) \psi(a)q_B\left(V_2\right)^* .
$$
An extension $\varphi : A \to Q(B)$ is \emph{split} when there is
a $*$-homomorphism $\pi : A \to M(B)$ such that $\varphi = q_B
\circ \pi$ and \emph{asymptotically split} when there is an
asymptotic homomorphism $\pi_t : A \to M(B), t \in [1,\infty)$,
such that $q_B \circ \pi_t = \varphi$ for all $t$. We say that
\text{$\Ext(A,B)$ is a group} when every extension $\varphi : A
\to Q(B)$ has an inverse, meaning that there is another extension
$\varphi' : A \to Q(B)$, \emph{the inverse of $\varphi$}, such
that $\varphi \oplus \varphi'$ is split. An extension $\varphi : A
\to Q(B)$ is \emph{semi-invertible} when there is another
extension $\varphi' : A \to Q(B)$ such that $\varphi \oplus
\varphi'$ is asymptotically split.

When the theory of $C^*$-extensions was first
introduced, in the work of Brown, Douglas and Fillmore, \cite{BDF1},
\cite{BDF2}, the authors had very good (operator theoretic) reasons
for wanting to trivialize the split extensions. \footnote{They also
  had good reasons for restricting the attention to essential
  extensions, but that's another story.} However, there are
other reasons why split extensions must be trivialized in order to
get a group from the semi-group $\Ext(A,B)$. For a split extension
$\psi$ it makes sense to define the direct sum $\psi^{\infty}$ of a
countably infinite collection of copies of $\psi$. Since $\psi \oplus
\psi^{\infty} \oplus 0 = \psi^{\infty} \oplus 0$ in $\Ext(A,B)$ this
shows that split extensions are trivial in any group-quotient of
$\Ext(A,B)$. It is not difficult to show that $\psi^{\infty}$ can
also be defined when the extension $\psi$ is asymptotically split. In
fact, this is possible as soon as the extension splits via a
discrete asymptotic homomorphism, e.g when it is quasi-diagonal.
But by using the real parameter for the asymptotic section it can
also be arranged that $\psi \oplus \psi^{\infty} \oplus 0$ becomes
unitarily equivalent to $\psi^{\infty} \oplus 0$. It follows that
also asymptotically split extensions must vanish in a
group-quotient of $\Ext(A,B)$. In fact, any group-quotient of
$\Ext(A,B)$ must factor through the cancellation semi-group of
$\Ext(A,B)$. In retrospect it seems therefore not particularly
surprising that it is not generally enough to trivialize only the
split extensions to get a group, or even the asymptotically split
extensions, as demonstrated in \cite{MT4}. In fact, seen through
the right looking-glasses it seems more surprising that
$\Ext(A,B)$ actually \emph{is} a group in so many cases, and that
semi-invertibility prevails in many cases where invertibility
fails.

Complementing on the cases
covered by the results in \cite{MT3}, \cite{MT1}, \cite{M},
\cite{Th4} and \cite{ST} we shall show in this paper that all
extensions in $\Ext(A,B)$ are semi-invertible when
\begin{enumerate}
\item[a)] $A$ is the reduced group $C^*$-algebra $C_r^*(G)$ and the
  group $G$ is an amalgamated free product $G = G_1
  *_F G_2$ with $F$ finite, $G_2$ is amenable and $G_1$ abelian, and when
\item[b)] $A$ is the amalgamated free product of $C^*$-algebras, $A  = A_1*_D A_2$, when
  $D$ is nuclear and all extensions of $A_i$ by $ B$ are semi-invertible,
  $i =1,2$.
\end{enumerate}
The result concerning a) is actually slightly more general and
involves a KK-theory condition which is automatically fullfilled when
$G_1$ is abelian. Furthermore we establish a few permanence properties for
semi-invertibility: If all extensions of $A$ and $A'$ by $B$ are
semi-invertible then so are all extensions of $A\oplus A'$ by $B$, all
extensions of $C(\mathbb T)\otimes A$ by $B$ and all extensions of
$\mathbb K \otimes A$ by $B$. It follows from this that all extensions
of $A$ by $B$ are semi-invertible when
\begin{enumerate}
 \item[a')] $A = C_r^*(G')$ provided $G' = \mathbb Z^k \times H \times
   G$ where $H$ is a finite group and $G$ is an amalgamated free product
   as in a) above, and when
\item[b')] $A$ is the full group $C^*$-algebra $C^*(\mathbb Z^k \times H \times G'')$ where $H$ is a
  finite group and $G''$ is obtained through successive amalgamations
$$
G'' = \left( \cdots \left(\left(G_1 *_{H_1} G_2\right) *_{H_2} G_3\right) *_{H_3}
\dots \dots \right) *_{H_{n-1}} G_n,
$$
provided all the groups $H_1,H_2, \dots,
H_{n-1}$ are amenable, and all extensions of $C^*(G_i)$ by $B$ are
semi-invertible, $i = 1,2, \dots,n$.
\end{enumerate}
While we know from \cite{HS}, \cite{HLSW} and \cite{Se} that there are
non-invertible extensions of $A$ by $B$ in many of the cases dealt
with in a), our ignorance concerning invertibility of the extensions
handled by b') is complete: There is no known example of an extension of a full group $C^*$-algebra by a
stable $C^*$-algebra which is not invertible.

The proof of a) above is an elaboration of the ideas developed in
\cite{M}, \cite{Th4} and \cite{ST}. In particular, the argument
uses the notion of strong homotopy of extensions and depends on
Lemma 4.3 in \cite{MT1}. In contrast the method of proof of b) is
new and does not use strong homotopy of extensions. Instead a key
step uses methods devised for the classification of $C^*$-algebras
by Lin, Dadarlat and Eilers. This difference in the proofs has
consequences for the conclusions we obtain; in case a) the inverse
(for semi-invertibility) can be chosen to be invertible while we
do not know if this is so in case b).

\emph{Acknowledgement.} The main part of this work was done during a stay of
          both authors at the Mathematische Forchungsinstitut in
          Oberwolfach in January 2010 in the framework of the `Research in Pairs' programme. We want to thank the MFO for the perfect
          working conditions.

\section{The reduced  group $C^*$-algebra of free products with
  amalgamation over a finite subgroup}

Throughout $A$ and $B$ are separable $C^*$-algebras and $B$ is
stable. Two extensions $\varphi,\varphi': A \to Q(B)$ are \emph{strongly
  homotopic} when there is a path $\psi_t, t \in [0,1]$, of extensions
$\psi_t : A \to Q(B)$ such that
\begin{enumerate}
\item[1)] $t \mapsto \psi_t(a)$ is continuous for all $a \in A$, and
\item[2)] $\psi_0 = \varphi$ and $\psi_1 = \varphi'$.
\end{enumerate}

By Lemma 4.3 of \cite{MT1} we have the following

\begin{thm}\label{basic} Assume that two extensions $\varphi,
  \varphi' : A \to Q(B)$ are strongly homotopic. Then $\varphi$ is asymptotically
  split if and only if $\varphi'$ is asymptotically split.
\end{thm}

 In some of the cases we deal with below we show that for
any extension $\varphi : A \to Q(B)$ there is an extension $\psi : A \to Q(B)$ such that $\varphi \oplus \psi$ is strongly
homotopic to a split extension. This will be expressed by saying that $\varphi$ is \emph{strongly
  homotopy invertible}. Thanks to Theorem \ref{basic} this implies that $\varphi$ is semi-invertible. In some cases it turns
out that $\psi$ can be taken to be
invertible. We express this by saying that $\varphi$ \emph{is strongly homotopy invertible with an
invertible inverse.}

\begin{lemma}\label{correction} Let $G_i, i = 1,2$, be discrete
  countable amenable groups with a common finite subgroup $H \subseteq G_i, i
  = 1,2$. Let $G_1 *_H G_2$ be the amalgamated free product group. Let $\mu : C^*(G_1 *_H G_2) \to
  C^*_r(G_1 *_H G_2)$ be the canonical surjection and let $h_\tau :
  C^*(G_1*_H G_2) \to
  \mathbb C$ be the character corresponding to the trivial
  one-dimensional representation of $G_1*_H G_2$. There are then a separable infinite-dimensional
  Hilbert space $\mathbb H$, $*$-homomorphisms $\sigma,\sigma_0 :
  C^*_r(G_1 *_H G_2) \to B(\mathbb H)$, and a
  path
$$
\zeta_s : C^*(G_1 *_H G_2) \to B(\mathbb H), \ s \in [0,1],
$$
of unital
  $*$-homomorphisms such that
\begin{enumerate}
\item[a)] $\zeta_0 = \sigma \circ \mu$;
\item[b)] $\zeta_1 = h_\tau \oplus\sigma_0\circ \mu$;
\item[c)] $\zeta_s(a) - \zeta_0(a) \in \mathbb K, \ s \in [0,1]$, and
\item[d)] $s \mapsto \zeta_s(a)$ is continuous for all $a \in
  C^*(G_1*_H G_2)$.
\end{enumerate}
\begin{proof} Set $G = G_1 *_H G_2$. Being amenable $ G_i $ has the
Haagerup Property. See the discussion in 1.2.6 of \cite{CCJJV}. It
follows then from Propositions 6.1.1 and 6.2.3 of \cite{CCJJV} that
also $ G $ has the Haagerup Property. Since the Haagerup Property
implies $K$-amenability by \cite{Tu} (or Theorem 1.2 in \cite{HK}) we conclude that $G$ is
$K$-amenable. We can therefore find a separable infinite-dimensional
  Hilbert space $\mathbb H$ and $*$-homomorphisms $\sigma, \sigma_0 :
  C^*_r(G) \to B(\mathbb H)$ such that $\sigma$ and $h_{\tau} \oplus
  \sigma_0$ are both unital and
\begin{enumerate}
\item[1)] $\sigma \circ \mu (x) - \left(h_{\tau} \oplus \sigma_0 \circ
    \mu\right)(x) \in \mathbb K, \ x \in C^*(G)$, and
\item[2)] $\left[\sigma \circ \mu , h_{\tau} \oplus \sigma_0 \circ \mu
  \right] = 0$ in $KK\left(C^*(G), \mathbb K\right)$,
\end{enumerate}
cf. \cite{C}. By adding the same unital and injective $*$-homomorphism to $\sigma$
and $\sigma_0$ we can arrange that both $\sigma$ and $\sigma_0$ are
injective and have no non-zero compact operator in their range. Since
$\mu|_{C^*(G_i)} : C^*(G_i) \to C^*_r(G_i)$ is injective because $G_i$
is amenable, it
follows that $\sigma \circ \mu|_{C^*(G_i)}$ and $(h_\tau \oplus \sigma_0
\circ \mu)|_{C^*(G_i)}$ are admissible in the sense of Section 3 of
\cite{DE} for each $i$. Thus Theorem 3.12 of
\cite{DE} applies to show that there is a
norm-continuous path $u^i_s, s\in [1,\infty)$, of unitaries in $1 +
\mathbb K$ such that
\begin{equation}\label{Hfin1}
\lim_{s \to \infty}
\left\|\sigma\circ \mu|_{C^*(G_i)}(a)  -
  u^i_s \left(h_\tau \oplus \sigma_0\circ
   \mu\right)|_{C^*(G_i)}(a){u^i_s}^* \right\| = 0
\end{equation}
for all $a \in
C^*(G_i)$ and
\begin{equation}\label{Hfin2}
\sigma \circ \mu|_{C^*(G_i)}(a)  -
  u^i_s \left(h_\tau \oplus \sigma_0\circ
   \mu\right)|_{C^*(G_i)}(a){u^i_s}^* \in \mathbb K
\end{equation}
for all $a
  \in C^*\left(G_i\right)$ and all $s \in [1,\infty)$. Set
$$
F = \left(h_\tau \oplus \sigma_0\circ \mu\right) \left(C^*(H)\right)
$$
which is a finite dimensional unital $C^*$-subalgebra of $B(\mathbb
H)$, and let $P :
  B(\mathbb H) \to F'\cap B(\mathbb H)$ be the conditional expectation
  given by
$$
P(x) = \int_{U(F)} uxu^* \ d u,
$$
where we integrate with respect to the Haar-measure on the unitary
group $U(F)$ of $F$. Note that $P\left(1 + \mathbb K\right) \subseteq
1 + \mathbb K$. It follows from (\ref{Hfin1}) that ${u^2_s}^*u^1_s$
asymptotically commutes with elements of $F$ and hence also that
\begin{equation}\label{applied}
\lim_{s \to \infty} \left\|P\left({u^2_s}^*u^1_s\right) -
  {u^2_s}^*u^1_s\right\| = 0.
\end{equation}
Standard $C^*$-algebra techniques provides us then with a
norm-continuous path $v_t, t \in [1,\infty)$, of unitaries in $F'\cap
(1+ \mathbb K)$ such that $\lim_{s \to \infty} \left\|v_s -
P\left({u^2_s}^*u^1_s\right)\right\| = 0$, which combined with
(\ref{applied}) implies that
$$
\lim_{s \to \infty} \left\|u^2_sv_s - u^1_s\right\| = 0 .
$$
It follows that we can work with $u^2_sv_s$ in the place of $u^1_s$ to
arrange that besides (\ref{Hfin1}) and (\ref{Hfin2}) we have also that
\begin{equation*}\label{Hfin3}
\Ad u^1_s \circ \left(h_\tau \oplus \sigma_0\circ \mu\right)|_{C^*(H)} = \Ad
u^2_s \circ \left(h_\tau \oplus \sigma_0\circ \mu\right)|_{C^*(H)}
\end{equation*}
for all $s$. It follows that the
$*$-homomorphisms
$$
\psi'_s = \left(\Ad u^1_s \circ \left(h_\tau \oplus \sigma_0\circ \mu\right)\right)
*_{C^*(H)} \left(\Ad u^2_s \circ \left(h_\tau \oplus \sigma_0\circ
  \mu\right)\right)
$$
are all defined and give us a
norm-continuous path of unital
$*$-homomorphisms $\eta_s : C^*(G) \to B(\mathbb H), \ s \in
[0,1]$, such
that
\begin{enumerate}
\item[a')] $\eta_0 = \left(\Ad u^1_1 \circ \left(h_\tau \oplus\sigma_0\circ
  \mu\right)\right)
  *_{C^*(H)} \left(\Ad u^2_1 \circ \left(h_\tau \oplus\sigma_0\circ \mu\right)\right) $;
\item[b')] $\eta_1 = \sigma \circ \mu$;
\item[c')] $\eta_s(a) - \eta_0(a) \in \mathbb K, \ a \in C^*(G), s \in [0,1]$.
\end{enumerate}
The unitary group of $F' \cap \left(\mathbb C1 + \mathbb K\right)$ is
norm-connected; a fact which can be seen either from the spectral
theory of compact operators or by observing that the algebra is AF. By
using first a continuous path of unitaries connecting ${u^2_1}^*u^1_1$ to $1$ in
$F' \cap \left(1 + \mathbb K\right)$ and then a continuous path of
unitaries connecting $u^2_1$ to $1$ in the unitary group of $1 +
\mathbb K$, we obtain continuous paths $w^1_s$ and $w^2_s$, $s \in [0,1]$, of
unitaries in $1 + \mathbb K$ such that $w^1_0 = w^2_0 = 1$,  $w^1_1 =
u^1_1$, $w^2_1 = u^2_1$ and $\Ad
w^1_s \circ \left(h_\tau \oplus\sigma_0\circ \mu\right)|_{C^*(H)} = \Ad
w^2_s \circ \left(h_\tau \oplus\sigma_0\circ \mu\right)|_{C^*(H)}$ for all $s \in [0,1]$. It follows that the
$*$-homomorphisms
$$
\eta'_s = \left(\Ad w^1_s \circ \left(h_\tau \oplus \sigma_0\circ \mu\right)\right)
*_{C^*(H)} \left(\Ad w^2_s \circ \left(h_\tau \oplus \sigma_0\circ
  \mu\right)\right)
$$
are all defined and give us
a norm-continuous path of unital
$*$-homomorphisms $\eta'_s : C^*(G) \to B(\mathbb H), \ s \in
[0,1]$, such
that
\begin{enumerate}
\item[a'')] $\eta'_0 = h_\tau \oplus \left(\sigma_0\circ \mu\right)$;
\item[b'')] $\eta'_1 = \left(\Ad u^1_1 \circ \left(h_\tau \oplus\sigma_0\circ
  \mu\right)\right)
  *_{C^*(H)} \left(\Ad u^2_1 \circ \left(h_\tau \oplus\sigma_0\circ \mu\right)\right) $;
\item[c'')] $\eta'_s(a) - \eta'_0(a) \in \mathbb K, \ a \in C^*(G), s \in [0,1]$.
\end{enumerate}
The desired path $\zeta$ is then obtained by concatenation of the
paths, $\eta$ and $\eta'$.
\end{proof}
\end{lemma}

\begin{thm}\label{thm1} Let $G_i, i = 1,2$, be discrete
  countable amenable groups with a common finite subgroup $H \subseteq G_i, i
  = 1,2$, and let $B$ be a separable stable $C^*$-algebra.  Let $G_1 *_H G_2$
  be the amalgamated free product group. Assume that the map
$$
i_1^* - i_2^* : KK\left(C^*(G_1),B\right) \oplus
KK\left(C^*(G_2),B\right) \to KK\left(C^*(H),B\right),
$$
induced by the inclusions $i_j : C^*(H) \to C^*(G_j), j = 1,2$, is rationally surjective, i.e. for every $x \in
KK\left(C^*(H),B\right)$ there is an $n \in \mathbb N \backslash
\{0\}$ such that $nx$ is in the range of $i_1^* - i_2^*$.

It follows that every extension of $C^*_r\left(G_1 *_H
  G_2\right)$ by $B$ is strongly homotopy invertible with an
invertible inverse.
\end{thm}
\begin{proof} Set $G = G_1 *_HG_2$ and consider an extension $\varphi
  : C^*_r\left(G_1*_HG_2\right) \to Q(B)$. Since $C^*(G) \simeq
  C^*\left(G_1\right) *_{C^*(H)} C^*\left(G_2\right)$ it follows from Proposition 2.8 of
  \cite{Th2} that every
  extension of $C^*(G)$ by $B$ is invertible. As observed in the proof
  of Lemma \ref{correction}, $G$ is $K$-amenable and it follows
  therefore from \cite{C} that $\mu^* : \Ext^{-1}\left(C^*_r(G),B\right) \to
  \Ext^{-1}\left(C^*(G),B\right)$ is an isomorphism. In particular the
  inverse of $\varphi \circ \mu$ is in the range of $\mu^*$, which
  means that there is an invertible extension $\varphi'' : C_r^*\left(G\right)
  \to Q(B)$ such that
\begin{equation}\label{eq102}
\left[\varphi \circ \mu \oplus \varphi''\circ \mu\right] = 0
\end{equation}
in $\Ext^{-1}\left(C^*(G),B\right)$. Let $\beta_0:
C^*_r(G) \to M(B)$ be an absorbing homomorphism, whose existence is guaranteed by
\cite{Th1} and set $\varphi' = \varphi \oplus q_B \circ \beta_0$. By
Lemma 2.2 of \cite{Th2} $\beta_0|_{C^*_r(G_i)} : C^*_r(G_i) \to M(B)$ is absorbing for each $i =
1,2$. Since $G_i$ is amenable $\mu|_{C^*(G_i)} : C^*(G_i) \to C^*_r(G_i)$ is a $*$-isomorphism and it
follows therefore from (\ref{eq102}) that
  $\left(\varphi' \circ \mu \oplus \varphi''\circ
    \mu\right)|_{C^*(G_i)}$ is a split extension for each
  $i$. In other words, there are $*$-homomorphisms $\pi_i :
  C^*\left(G_i\right) \to M(B)$ such that $\left(\varphi' \circ \mu \oplus \varphi''\circ
    \mu\right)|_{C^*(G_i)} = q_B \circ \pi_i, i =
  1,2$. Note that
$$
\pi_1(x) - \pi_2(x) \in B
$$
for all $x \in C^*(H)$ so that $\left(\pi_1,\pi_2\right)$ represents an
element of $KK\left(C^*(H),B\right)$. We need to change the
situation to a case where this pair represents $0$ in
$KK\left(C^*(H),B\right)$. This is done as follows:

$\beta_0|_{C^*(G_i)}, i = 1,2$, are both absorbing so after adding
$q_B \circ \beta_0$ to $\varphi''$ we get a situation
where there are unitaries $u_i \in M(B)$ such that $\Ad u_i \circ
\pi_i(y) - \beta_0(y) \in B$ for all $y \in C^*(G_i), i = 1,2$. Then
$$
\varphi' \circ \mu \oplus \varphi''\circ
    \mu = Ad q_B(u_2^*) \circ \left(\left( q_B \circ \Ad u_2u_1^* \circ
    \beta_0|_{C^*(G_1)}\right) *_{C^*(H)} \left(q_B \circ \beta_0|_{C^*(G_2)}
  \right)\right) .
$$
It follows that we can choose the lifts, $\pi_1,\pi_2$, above such
that $\left[\pi_1|_{C^*(H)},\pi_2|_{C^*(H)}\right] = \left[\Ad w \circ
  \beta_0|_{C^*(H)}, \beta_0|_{C^*(H)} \right]$ in $KK\left(C^*(H),B\right)$
  where $w = u_2u_1^*$. To proceed we need a description of the
  KK-groups obtained in \cite{Th1} and \cite{Th3}: When $A$ is a
  separable $C^*$-algebra
 and $\alpha: A \to
  M(B)$ is an absorbing $*$-homomorphism, there is an isomorphism between $K_1\left(\mathcal
    D_{\alpha}(A)\right)$ and $KK(A,B)$, where
\begin{equation}\label{xx12}
\mathcal
    D_{\alpha}(A) = \left\{ m \in M(B) : \ \alpha(a)m - m \alpha(a)
      \in B \ \forall a \in A \right\}.
\end{equation}
The isomorphism sends a unitary $u \in \mathcal
    D_{\alpha}(A)$ to $\left[\Ad u \circ \alpha,
      \alpha\right]$. Ignoring the passage to matrices in $K_1$ our assumption
    implies, in this picture of KK-theory, that there is an $n > 0$
    and a
  norm-continuous path of unitaries in $\mathcal
  D_{\beta_0}\left(C^*(H)\right)$ connecting $w^n$ to a product
  $w_2^*w_1$, where $w_i \in  \mathcal
  D_{\beta_0}\left(C^*(G_i)\right), i = 1,2$. Then $\left[\Ad w^n \circ
  \beta_0|_{C^*(H)}, \beta_0|_{C^*(H)} \right] = \left[\Ad w_1 \circ
  \beta_0|_{C^*(H)}, \Ad w_2 \circ \beta_0|_{C^*(H)} \right]$ in $KK(C^*(H),B)$. Note that
$$
q_B \circ \beta_0 \circ \mu = \left(q_B \circ \Ad w_1^* \circ \beta_0|_{C^*(G_1)}\right)
*_{C^*(H)} \left(q_B \circ \Ad w_2^* \circ \beta_0|_{C^*(G_2)}\right) .
$$
After adding
$$
 \underset{n-1 \ \text{times}}{\underbrace{\left(\varphi'  \oplus \varphi''
    \right) \oplus \left(\varphi'  \oplus \varphi''
    \right) \oplus \dots \oplus \left(\varphi'  \oplus \varphi''
    \right)} } \oplus q_B \circ \beta_0
$$
to $\varphi''$ we come in a
position where the pair $(\pi_1,\pi_2)$ can be chosen such that
$\left[\pi_1,\pi_2\right] = 0$ in $KK\left(C^*(H),B\right)$. (If we take the
passage to matrices in $K_1$ into account in the previous
argument, it may be necessary to add a finite direct sum of copies of
$q_B \circ \beta_0$ instead of a single copy.)

We can then proceed as follows: Set $\beta =
q_B \circ \beta_0^{\infty}$ where $\beta_0^{\infty}$ is the direct sum
of a sequence of copies of $\beta_0$. By adding $\beta$ to
$\varphi''$ we come then in a situation where Theorem 3.8 of
\cite{DE} applies to give us a continuous path $u_t, t \in
[1,\infty)$, of unitaries in $1 + B$ such that
$$
\lim_{t \to \infty} \Ad u_t \circ \pi_1(x)  = \pi_2(x)
$$
for all $x \in C^*(H)$. Since $C^*(H)$ is finite dimensional we have
that for $t$ large enough there is a unitary $v \in 1 + B$ such that
$vu_t\pi_1(x)u_t^*v^* = \pi_2(x)$ for all $x \in C^*(H)$. Hence, by
exchanging $\pi_1$ with $\Ad vu_t \circ \pi_1$ we conclude that
$\varphi'\circ \mu \oplus \varphi'' \circ \mu$ is split. By a standard
argument, based on Kasparov's stabilization theorem, we may add a split
extension to arrange that $\varphi'\circ \mu \oplus \varphi'' \circ
\mu = q_B \circ \chi \oplus 0$ where $\chi : C^*(G) \to M(B)$ is
a unital $*$-homomorphism. Let $\gamma : G \to M(B)$ be the unitary
representation of $G$ defined by $\chi$ and let
$\zeta_s$ be the continuous path of $*$-homomorphisms
from Lemma \ref{correction}, and $\nu_s$ the corresponding unitary
representations. Let $h_{\gamma \otimes \nu_s}$ be the $*$-homomorphism $C^*(G) \to M(B)$
  defined from the tensor product representation $\gamma \otimes \nu_s$ by use of a spatial isomorphism
  $B \otimes \mathbb K \simeq B$. Then
$$
q_B  \circ h_{\gamma \otimes \nu_s}, \ s \in [0,1],
$$
is a strong homotopy of extensions of $C^*\left(G\right)$ by $B$.
By the argument used in the proof of Theorem 2.3 of \cite{Th3} and
again in the proof of Theorem 2.2 in \cite{ST} the properties of
$\left\{\zeta_s\right\}$ ensure that this homotopy factors through
$C^*_r(G)$ and gives us a strong homotopy, as well as split
extensions $\psi,\psi'$, of $C^*_r(G)$ by $B$ connecting $\varphi
\oplus q_B \circ \beta_0 \oplus \varphi'' \oplus \psi  =\varphi'
\oplus \varphi'' \oplus \psi $ to $\psi'$. Since $q_B \circ
\beta_0 \oplus \varphi'' \oplus \psi$ is invertible, this
completes the proof.
\end{proof}

As in \cite{ST} the fact that the strong homotopy inverse is
invertible implies that the group $\Ext^{-1/2}(C^*_r(G_1 *_H G_2),B)$
of extensions modulo asymptotically split extensions agrees with the
corresponding KK-theory group and can be calculated from the universal
coefficient theorem. The proof is the same as in \cite{ST} and we omit
it here.

The KK-condition of Theorem \ref{thm1} is satisfied when $G_1$ is
abelian since in this case already
the map
$$
i_1^* : KK\left(C^*\left(G_1\right),B\right)  \to KK(C^*(H),B)
$$
is surjective. This follows because there is in this case a $*$-homomorphism
$p : C^*\left(G_1\right) \to C^*(H)$ which is a left-inverse for $i_1$. We get in this way the following corollary.

\begin{cor}\label{abelcor} Let $G_1$ and $G_2$ be countable discrete
  amenable groups with a common finite subgroup $H \subseteq G_i, i
  = 1,2$, and $B$ a separable stable $C^*$-algebra.  Let $G_1 *_H G_2$
  be the amalgamated free product group. Assume that $G_1$ is abelian. It follows that every extension of $C^*_r\left(G_1 *_H
  G_2\right)$ by $B$ is strongly homotopy invertible with an
invertible inverse.
\end{cor}

\begin{example}\label{Sl(2,Z)} It is known that
$$
Sl_2(\mathbb Z) \simeq \mathbb Z_4 *_{\mathbb Z_2} \mathbb Z_6 ,
$$
cf. p. 11 in \cite{S}. Hence Corollary \ref{abelcor} applies.
(As the generator of $\mathbb Z_4$ one can use
$\left(\begin{smallmatrix} 0 & -1 \\ 1 & 0 \end{smallmatrix} \right)$,
and $\left(\begin{smallmatrix} 1 & -1 \\ 1 & 0 \end{smallmatrix}
\right)$ can serve as the generator of $\mathbb Z_6$. The amalgamation
is over the subgroup $\pm 1$.) It has been shown by Hadwin and Shen in
Corollary 4.4 of \cite{HS} that
one can get an example of an non-invertible extension of
$C^*_r\left(Sl_2(\mathbb Z)\right)$ by $\mathbb K$, starting from the
non-invertible extension of $C^*_r\left(\mathbb F_2\right)$ found by
Haagerup and Thorbj\o rnsen in \cite{HT}. This means that concerning
invertibility of extensions of $C^*_r\left(Sl_2(\mathbb Z)\right)$ the
situation is as for $C_r^*\left(\mathbb F_2\right)$: For every
stabilization $B$ of a unital separable $C^*$-algebra there are
non-invertible extensions of $C^*_r\left(Sl_2(\mathbb Z)\right)$ by $B$, but all are semi-invertible. And the
inverse (for semi-invertibility) can be taken to be invertible.

For the full group $C^*$-algebra $C^*\left(Sl_2(\mathbb Z)\right)$ the
situation is also as for $\mathbb F_2$, namely that all extensions by
$C^*\left( Sl_2(\mathbb Z)\right)$ are invertible. This follows from
\cite{Br} when the ideal is $\mathbb K$ and from \cite{Th2} when it is
an
arbitrary separable stable $C^*$-algebra.

\end{example}

\begin{remark}
The KK-condition of Theorem \ref{thm1} can fail even
when $G_1$ and $G_2$ are finite and equal, and $H$ is abelian. Here is the simplest example.
Let $\alpha$ be the unique non-trivial automorphism of $\mathbb
Z_3$ which has order 2 and let $G_1=\mathbb Z_3\rtimes_\alpha\mathbb Z_2$ be the
semidirect product by this automorphism. Thus $G_1$ is a copy of the
symmetric group $S_3$. Set
$H=\mathbb Z_3\subset G_1$. Let $B=\mathbb K$. Then
$KK(C^*(G),B)\cong R(G)$ for any finite group $G$, where $R(G)$
denotes the Grothendieck group of the semigroup generated by
irreducible (necessarily finite dimensional) representations of
$G$. The functorial map $KK(C^*(G_1),B)\to KK(C^*(H),B)$ becomes the
restriction map $R(G_1)\to R(H)$ after the
above identification. The abelian group $R(H)$ is freely generated by the three
one-dimensional representations, $\rho_0$, $\rho_1$ and $\rho_2$,
that send a fixed generator of $H$ to 1, $e^{2\pi i/3}$ and
$e^{-2\pi i/3}$, respectively. As the number of irreducible
representations equals the number of conjugacy classes by the Burnside
theorem, and as the group order equals the sum of squares of the
dimensions of these representations, it follows that $G_1$ has three
irreducible representations; two, $\sigma_0$ and $\sigma_1$, of
dimension 1 and one, $\tau$, of dimension 2. Thus, $R(G_1)$ is
freely generated by three representations, $\sigma_0$, $\sigma_1$
and $\tau$. One of the one-dimensional representations,
$\sigma_0$, is the identity one, and the other, $\sigma_1$, maps
$H$ to 1 and $G_1\setminus H$ to $-1$. Restrictions of both to $H$
equal the trivial representation $\rho_0$ of $H$. The
two-dimensional representation $\tau$ is the orthogonal complement
to the constant functions in the obvious representation of $G_1$
on $l^2(H)\cong\mathbb C^3$. Then it is easy to see that
$\tau|_H=\rho_1\oplus\rho_2$. Thus, the restriction map $R(G_1)\to
R(H)$ is not surjective.

This example goes only to show that the KK-condition of Theorem \ref{thm1}
is not vacuous. For all we know the conclusion of Theorem
\ref{thm1} may very well be true without this condition.

\end{remark}

\section{Amalgamated free product $C^*$-algebras}

In this section we consider free products of $C^*$-algebras with
amalgamation. The first result is an application of the relative
K-homology developed by the authors in \cite{MT2}.

\begin{thm}\label{amalthm} Let $A_1,A_2$ and $B$ be separable
  $C^*$-algebras, $B$ stable. Let $D$ be a common $C^*$-subalgebra of
  $A_1$ and $A_2$, i.e. $D \subseteq A_1$ and $D \subseteq
  A_2$. Assume that
\begin{enumerate}
\item[1)] there is a $*$-homomorphism $\alpha_0 : A_1 *_D A_2
  \to M(B)$ such that also $\alpha_0|_{A_1}, \alpha_0|_{A_2}$ and
  $\alpha_0|_D$ are absorbing, and assume that
\item[2)] $\Ext(A_1,B)$ and $\Ext(A_2,B)$ are both groups.
\end{enumerate}
It follows that every extension of $A_1 *_D A_2$ by $B$ is strongly homotopy invertible.
\end{thm}
\begin{proof} Set $\alpha = q_B \circ \alpha_0$ and consider an
  extension $\varphi : A_1 *_D A_2 \to Q(B)$. By assumption 2) there
  is an extension $\psi_i : A_i \to Q(B)$ representing
  the inverse of $\varphi|_{A_i}$ in $\Ext(A_i,B)$ both for $i=1$ and $i=2$. Then $\psi_1|_D$
  and $\psi_2|_D$ represent the same element in $\Ext(D,B)$, namely
  the inverse of the element represented by $\varphi|_D$. After
  addition of $\alpha_0|_{A_i}$ to $\varphi|_{A_i}$ we therefore assume
  that $\psi_1|_D$ and $\psi_2|_D$ are unitarily
  equivalent. Thus, after conjugating $\psi_2$ by a unitary, we can
  arrange that $\psi_1|_D = \psi_2|_D$. Then $\psi = \psi_1 *_D
  \psi_2 : A_1 *_D A_2 \to Q(B)$ is defined. Set $\Phi = \varphi
  \oplus \psi$. By adding a copy of $\alpha$ to $\Phi$ both
  extensions $\Phi|_{A_i} : A_i \to Q(B), i = 1,2$,
  become split, i.e. there are $*$-homomorphisms $\Phi_i : A_i \to M(B)$
  such that $q_B \circ \Phi_i = \Phi|_{A_i}, i = 1,2$. By passing to a
  unitarily equivalent extension, i.e. by conjugating $\Phi$ by a unitary of
  the form $q_B(u)$, we can arrange that in addition $q_B \circ \Phi_2 =
  \alpha|_{A_2}$ and that $\Phi_2 = \alpha_0|_{A_2}$. Then $q_B \circ \Phi_1$
  represents an element of the relative extension semi-group
  $\Ext_{D,\alpha|_{A_1}}\left(A_1,B\right)$, cf. \cite{MT2}. In fact, it
  follows from Lemma 3.2 of \cite{MT2} and assumption 2) that $q_B \circ \Phi_1$ is
  invertible in this semi-group, i.e. $q_B \circ \Phi_1 \in
  \Ext_{D,\alpha|_{A_1}}^{-1}\left(A_1,B\right)$. Let $\Phi_1' : A_1
  \to Q(B)$
  represent the inverse of $q_B \circ \Phi_1$ in
  $\Ext^{-1}_{D,\alpha|_{A_1}}\left(A_1,B\right)$ and note that
  $\Phi_1' *_D \alpha|_{A_2} : A_1 *_D A_2 \to Q(B)$ is then
  defined. After addition by this extension to $\Phi$ we can assume that
  $\Phi_1$ represents $0$ in
  $\Ext^{-1}_{D,\alpha|_{A_1}}\left(A_1,B\right)$. By definition of
  $\Ext_{D,\alpha|_{A_1}}\left(A_1,B\right)$ this means that
  there is a unitary $u$ in the connected component of $1$ in the relative commutant of $\alpha(D)$ in
  $Q(B)$ such that $\Ad u \circ q_B \circ \Phi_1 = \alpha|_{A_1}$. Let
  $u_t, t \in [0,1]$, be a continuous path of unitaries in $\alpha(D)'
  \cap Q(B)$ such that $u_0 = 1$ and $u_1 = u$. Then
$$
\psi_t = \left( \Ad u_t \circ q_B \circ \Phi_1\right) *_D \left(q_B \circ
\Phi_2\right)
 $$
is defined for every $t \in [0,1]$, and $\psi_t, t \in [0,1]$, is a
strong homotopy of extensions connecting $\Phi = \psi_0$ to $\psi_1 = q_B \circ \alpha$. This completes the proof.
\end{proof}

Condition 1) of Theorem \ref{amalthm} is always satisfied when $D$ is
nuclear or is the range of a conditional expectation $A_i \to D$ for
both $i =1$ and $i =2$, but it can fail in general. See
\cite{Th2}. Condition 2) is satisfied when $A_1$ and $A_2$ are
nuclear so Theorem \ref{amalthm} has the following corollary.

\begin{cor}\label{nuclearfree} Let $A_1,A_2$ and $B$ be separable
  $C^*$-algebras, $B$ stable. Let $D$ be a common $C^*$-subalgebra of
  $A_1$ and $A_2$, i.e. $D \subseteq A_1$ and $D \subseteq
  A_2$. If $A_1, A_2$ and $D$ are all nuclear it follows that every
  extension of $A_1*_DA_2$ by $B$ is strongly homotopy invertible.
\end{cor}

The next theorem shows that condition 2) of Theorem \ref{amalthm} can
be weakened when $D$ is nuclear, at the price of a slightly weaker conclusion.

\begin{thm}\label{amalsemi3}
Let $A_1,A_2$ and $B$ be separable
  $C^*$-algebras, $B$ stable. Let $D$ be a common $C^*$-subalgebra of
  $A_1$ and $A_2$, i.e. $D \subseteq A_1$ and $D \subseteq
  A_2$. Assume that
\begin{enumerate}
\item[1)] there is a $*$-homomorphism $\beta : A_1 *_D A_2
  \to M(B)$ such that $\beta|_D : D \to M(B)$ is absorbing,
\item[2)] that $\Ext(D,B)$ and
  $\Ext\left(D,C_0\left([1,\infty),B\right)\right)$ are both groups,
  and
\item[3)] that all extensions of $A_1$ by $B$ and all extensions of $A_2$
  by $B$ are semi-invertible.
\end{enumerate}

It follows that all extensions of $A_1 *_D A_2$ by $B$ are semi-invertible.
\end{thm}

\begin{proof} By adding units to $A_1$, $A_2$ and $D$ if necessary, we may assume that $D$ is unital.

1. step: (Finding the first candidate for the inverse.)

Let $\varphi : A_1 *_D A_2 \to Q(B)$ be an extension. By
  assumption 2) there are extensions $\psi_i : A_i \to Q(B)$ such that
  $\varphi|_{A_i} \oplus \psi_i : A_i \to Q(B)$ are asymptotically
  split, $i = 1,2$. By assumption 2) $\Ext(D,B)$ is a group and
  hence $\left[\psi_1|_D\right] = \left[\psi_2|_D \right] = -\left[\varphi|_D\right]$ in
  $\Ext(D,B) $. (There are various ways to see this; it follows for
  example from Lemma 4.7 of \cite{MT1}.) Furthermore, by assumption 1)
  there is a $*$-homomorphism $\beta: A_1 *_D
  A_2 \to M(B)$ such that $\beta|_D$ is absorbing. So
  after adding by $q_B \circ \beta|_{A_1}$ to $\psi_1$ and $q_B \circ
  \beta|_{A_2}$ to $\psi_2$ we may assume that ${\psi_1}|_D$ and $\psi_2|_D$ are unitarily
    equivalent, and hence without loss of generality that $\psi_1|_D =
    \psi_2|_D$. Then we have a candidate for a semi-inverse to
    $\varphi$, namely $\psi_1*_D \psi_2$. We will show that after
    addition by additional extensions (some of which may be non-trivial), $\varphi \oplus \left(\psi_1 *_D
      \psi_2\right)$ becomes asymptotically split.

2. step: (Removing a KK-obstruction.)

First note that $\varphi
    \oplus \left(\psi_1 *_D \psi_2\right)$ is split over $D$. Hence,
    by adding a copy of $q_B \circ \beta$ to $\varphi$ and conjugating
    by a unitary we can arrange that
\begin{equation}\label{betaonD}
\varphi
    \oplus \left(\psi_1 *_D \psi_2\right)|_D = q_B \circ \beta|_D.
\end{equation}
Let $\xi^i : A_i \to M(B)$, be equi-continuous asymptotic homomorphisms such
    that $q_B \circ \xi^i_t = \varphi|_{A_i} \oplus \psi_i$ for all
    $t$, $i = 1,2$. Note that by (\ref{betaonD}) we have that
\begin{equation}\label{diff}
\xi^i_t(d) - \beta(d) \in B
\end{equation}
for all $t \in [1,\infty), d \in D, i = 1,2$. Let $\beta^{\infty}$ denote the direct sum of a countable infinite
number of copies of $\beta$ and set $\pi = 1_{C_0[1,\infty)}\otimes
\beta^{\infty}$; i.e. $1_{C_0[1,\infty)}$ is the unit in the
multiplier algebra $M\left(C_0[1,\infty)\right)$ and $\pi(x) =
1_{C_0[1,\infty)} \otimes \beta^{\infty}(x) \in M\left(C_0[1,\infty),B\right)$. Then $\pi : D \to
M\left(C_0[1,\infty),B\right)$ is absorbing by Lemma 2.3 of \cite{Th3}. Since
$\Ext\left(D, C_0[1,\infty),B)\right)$ is the trivial group by
assumption 2), this implies that there is
a strictly continuous path $U_t, t \in [1,\infty)$, of unitaries in
$M(B)$ such that
\begin{equation}\label{nu1}
t \mapsto U_t\left(\xi^1_t(d) \oplus \beta^{\infty}(d)\right)U_t^* -
\left(\xi^2_t(d) \oplus \beta^{\infty}(d)\right)
\end{equation}
is in $C_0[1,\infty),B)$ for all $d \in D$. For each $n \in \mathbb
N$, $U_t,  t \in [1,n]$, defines a unitary $W_n$ in $M\left(C[1,n]
  \otimes B)\right)$ in the natural way. Set $\pi_n = 1_{C[1,n]}
\otimes \beta^{\infty}|_D$ and $\beta_n = 1_{C[1,n]}
\otimes \beta|_D$. Then (\ref{nu1}) and (\ref{diff}) imply that
\begin{equation}\label{uuu}
W_n \left ( {\beta_n} \oplus \pi_n\right)(d)W_n^*  -  \left ( {\beta_n} \oplus \pi_n\right)(d)  \in   C[1,n] \otimes B
\end{equation}
for all $d \in D$, i.e. $W_n$ is a unitary in the $C^*$-algebra
$\mathcal D_{\beta_n \oplus \pi_n}(D)$, cf. (\ref{xx12}).
Note that ${\beta_n}
  \oplus \pi_n$ is absorbing, again by Lemma 2.3 of \cite{Th3}, so that $K_1\left(\mathcal D_{\beta_n \oplus \pi_n}(D)\right) =
  KK(D,C[1,n]\otimes B)$ by (3.2) of \cite{Th3}. Identifying
  $KK(D,C[1,n]\otimes B)$ and $KK(D,B)$ we can say that
\begin{equation}\label{KKequ}
\left[\Ad W_n \circ \left({\beta_n} \oplus \pi_n \right) ,\left({\beta_n} \oplus
    \pi_n \right)\right] = \left[\Ad U_1 \circ \left(\beta|_D \oplus
    \beta^{\infty}|_D \right) ,\left(\beta|_D \oplus
    \beta^{\infty}|_D\right) \right].
\end{equation}
in $KK(D,C[1,n]\otimes B)$. Add then the extension
$$
\left(q_B \circ \Ad U_1 \circ (\beta \oplus \beta^{\infty})|_{A_1}\right) *_D \left(q_B
\circ (\beta \oplus \beta^{\infty})|_{A_2}\right)
$$
to $\varphi \oplus \left(\psi_1 *_D
      \psi_2\right)$. We can then exchange $\xi^1_t$ by $\xi^1_t
    \oplus \Ad U_1 \circ \left(\beta \oplus \beta^{\infty}\right)|_{A_1}$,
    $\xi^2_t$ by $\xi^2_t
    \oplus \left(\beta \oplus \beta^{\infty}\right)|_{A_2}$, and $U_t$ by
    $U_t \oplus U_1^*$. We may therefore return to the previous
    notation and conclude from (\ref{KKequ}) that
$$
\left[\Ad W_n \circ \left({\beta_n} \oplus \pi_n \right) ,\left({\beta_n} \oplus
    \pi_n \right)\right] =  0
$$
in $KK(D,C[1,n] \otimes B)$ for all $n$. It follows therefore that
$\diag (W_n,1,1, \dots, 1)$ is in
the connected component of $1$ in the unitary group of
$M_{k_n}\left(\mathcal D_{\beta_n \oplus \pi_n}(D)\right)$ for some
$k_n \in \mathbb N, \ k_n \geq 2$. Since $\beta_n \oplus \pi_n$ is absorbing, there
is an isomorphism from $M_{k_n}\left(\mathcal D_{\beta_n \oplus
    \pi_n}(D)\right)$ onto $M_{2}\left(\mathcal D_{\beta_n \oplus
    \pi_n}(D)\right)$ which takes $\diag (W_n,1,1,\dots,1)$ to
$\diag(W_n,1)$. It follows that $\diag (W_n,1)$ is in
the connected component of $1$ in the unitary group of
$M_{2}\left(\mathcal D_{\beta_n \oplus \pi_n}(D)\right)$ for each
$n$. After addition by the split extension $\beta^{\infty}$ so that we
can substitute $W_n \oplus 1$ for $W_n$, we may
therefore assume that $W_n$ is in the connected component of $1$ in the unitary group of
$\mathcal D_{\beta_n \oplus \pi_n}(D)$ for each $n \in \mathbb N$.

3. step: (The tricky part. This is an elaboration on ideas developed
by Lin, Dadarlat and Eilers, in \cite{L}, \cite{DE}, and a very similar argument was used to
prove Theorem 4.1 in \cite{Th3}.)

 Let $E_n$ denote the $C^*$-subalgebra of $M( C[1,n] \otimes B))$
 generated by the unit $1_{C[0,1]\otimes B}$, $C[1,n] \otimes B$ and $\left(\beta_n \oplus
   \pi_n\right)(D)$. It follows from (\ref{uuu}) that $\Ad W_n$
defines an automorphism $\alpha_n$ of $E_n$, and the path of unitaries
in $\mathcal D_{\beta_n \oplus \pi_n}(D)$ connecting $W_n$ to $1$ gives us a uniform
norm-continuous path of automorphisms in $\Aut E_n$ connecting
$\alpha_n$ to the identity in $\Aut E_n$. Since $E_n$ is separable, it
follows from 8.7.8 and 8.6.12 in \cite{P}, cf. Proposition 2.15 of
\cite{DE}, that $\alpha_n$ is asymptotically inner, i.e. there is a
continuous path $V^n_s, s \in [1,\infty)$, of unitaries in $E_n$ such
that $\alpha_n(x) = \lim_{s \to \infty} V^n_s x{V^n_s}^*$ for all $x
\in E_n$.

Let $F_1 \subseteq F_2 \subseteq F_3 \subseteq \cdots$ be a sequence of finite subsets with dense union in $D$. Since
$$
\lim_{s \to \infty} \sup_{t \in [1,n]} \|V^n_s(t)\left ( \xi^1_t \oplus
  \beta^{\infty}|_D \right)(d)  {V^n_s(t)}^* - U_t\left ( \xi^1_t \oplus
  \beta^{\infty}|_D \right)(d)  {U_t}^* \| = 0
$$
for all $d \in D$, we can find an $s_n \in [1,\infty)$ so big that
\begin{equation}\label{2u}
\|V^n_s(t)\left ( \xi^1_t \oplus
  \beta^{\infty}|_D\right)(d)  {V^n_s(t)}^* - U_t\left ( \xi^1_t \oplus
  \beta^{\infty}|_D \right)(d)  {U_t}^* \|  \leq \frac{1}{n}
\end{equation}
for all $s \geq s_n$, all $t \in [1,n]$ and all $d \in F_n$. Note that
\begin{equation}\label{uuuu}
\lim_{s \to \infty} {V^{n+1}_s(n)}^* V^n_s(n)x {V^n_s(n)}^*V^{n+1}_s(n) = x
\end{equation}
for all $x \in B \cup \left(\xi^1_t \oplus \beta^{\infty}\right)(D), t \in [1,n]$. To simplify notation, set $\Delta^k_s = {V^{k+1}_s(k)}^* V^{k}_s(k)$. It follows from (\ref{uuuu}) that if we increase $s_n$ we can arrange that
\begin{equation}\label{1u}
\| \Delta^k_s  \left ( \xi^1_t \oplus
  \beta^{\infty}|_D \right)(d)  {\Delta^k_s}^* - \left ( \xi^1_t \oplus
  \beta^{\infty}|_D \right)(d)  \| \leq \frac{1}{n^2}
\end{equation}
for all $d \in F_n, t \in [1,n]$, and all $k = 2,3, \cdots ,n$, when
$s \geq s_n$. Proceeding inductively we can arrange that $s_n <
s_{n+1}$ for all $n$. Let $s : [1,\infty) \to [1,\infty)$ be a
continuous increasing function such that $s(n) = s_{n+1}, n =
1,2,3,\cdots $. Define a norm-continuous path $W_t, t \in [1,\infty)$,
in
$$
E = C^*\left(1_B, \left(\xi^1_1 \oplus \beta^{\infty}|_D\right)(D),B
\right) = C^*\left(1_B, \left(\beta \oplus \beta^{\infty}|_D\right)(D),B
\right)
$$
such that $W_t = V^2_{s(t)}(t), t \in [1,2]$, and $W_t = V^{k+1}_{s(t)}(t) \Delta^{k}_{s(t)} \cdots  \Delta^3_{s(t)}\Delta^2_{s(t)}, \ t \in [k,k+1], \ k \geq 2$. Let $d \in F_n$ and consider $t \in [k,k+1]$, where $k \geq n$. Since $s(t) \geq s_{k+1}$ and $d \in F_{k+1}$, it follows from (\ref{1u}) that
\begin{equation}\label{4u}
W_t \left ( \xi^1_t \oplus
  \beta^{\infty}|_D \right)( d) W_t^*   \ \sim_{k \cdot \frac{1}{k^2}} \ V^{k+1}_{s(t)}(t)\left ( \xi^1_t \oplus
  \beta^{\infty}|_D \right)( d) {V^{k+1}_{s(t)}}(t)^* ,
\end{equation}
where $\sim_{\delta}$ means that the distance between the two elements
is at most $\delta$. Furthermore, it follows from (\ref{2u}) that
\begin{equation}\label{3u}
  V^{k+1}_{s(t)}(t) \left ( \xi^1_t \oplus
  \beta^{\infty}|_D \right)( d)  {V^{k+1}_{s(t)}}(t)^* \ \sim_{\frac{1}{k}} \ U_t \left ( \xi^1_t \oplus
  \beta^{\infty}|_D\right)( d) U_t^*  .
\end{equation}
It follows from (\ref{3u}), (\ref{4u}) and (\ref{nu1}) that
\begin{equation}\label{7u}
\lim_{t \to \infty} W_t \left ( \xi^1_t \oplus
  \beta^{\infty}|_D \right)( d)W_t^* - \left ( \xi^2_t \oplus
  \beta^{\infty}|_D \right)( d)   = 0,
\end{equation}
first when $d \in F_n$, and then for all $d \in D$ since $n$ was
arbitrary and $\left\{\xi^i_t\right\}_{i,t}$ equi-continuous.

Recall that $D$ is unital. For each $t$ there are unique
elements $x_t \in D, \lambda_t \in \mathbb C$ and $b_t \in B$ such that
$$
W_t = \left(\xi^1_t\oplus
  \beta^{\infty}|_D\right)(x_t) + \lambda_t { \left(\xi_t^1\oplus
  \beta^{\infty}|_D\right)(1)}^{\perp} + b_t .
$$
Since $q_B \circ \left ( \xi^1_t \oplus
  \beta^{\infty}|_D \right) = q_B \circ \left ( \xi^1_1 \oplus
  \beta^{\infty}|_D \right)$ is injective we find
that $\{x_t\}$ must be a continuous path of unitaries in $D$ such that
$\lim_{t \to \infty} x_tdx_t^* = d$ for all $d \in D$. Set
$$
U_t = W_t\left ( \xi^1_t \oplus
  \beta^{\infty}|_D \right)(x_t)^* + W_t\overline{\lambda_t}\left ( \xi^1_t \oplus
  \beta^{\infty}|_D \right)(1)^{\perp} .
$$
Then $U_t, t \in [1,\infty)$, is a continuous path of unitaries $1 +
B$ such that
$$
\lim_{t \to \infty} U_t\left ( \xi^1_t \oplus
  \beta^{\infty}|_D \right)(d) U_t^* - \left ( \xi^2_t \oplus
  \beta^{\infty}|_D \right)(d) = 0
$$
for all $d \in D$.

4. step: (Conclusion.)

By adding the split extension $q_B \circ
\beta^{\infty}$ we can now return to the notation in the 1. step and
assume that $U_t, t \in [1,\infty)$, is a continuous path of unitaries $1 +
B$ such that
\begin{equation}\label{nyref}
\lim_{t \to \infty} U_t\xi^1_t(d)U_t^* -  \xi^2_t(d)  = 0
\end{equation}
for all $d \in D$. Set
$$
\mathcal A =  \left\{ f \in C_b\left([1,\infty),M(B)\right) : \ f(1)
  -f(t) \in B \ \forall t \in [1,\infty) \right\}
$$
and note that $C_0\left([1,\infty),B\right)$ is an ideal in $\mathcal
A$. Let
$$
p : \mathcal A \to \mathcal A/C_0\left([1,\infty),B\right)
$$
be the quotient map. Define $*$-homomorphisms $\kappa_1 : A_1 \to
\mathcal A$ and $\kappa_2 : A_2 \to \mathcal A$ such that
$\kappa_1(a)(t) = U_t\xi^1_t(a)U_t^*$ and $\kappa_2(a)(t) = \xi^2_t(a)$,
respectively. Since $U_t\xi^1_t(d)U_t^* - \xi^2_t(d) \in D$ for all $t$ and
$d \in D$, it follows from (\ref{nyref}) that
$$
\left(p \circ \kappa_1\right) *_D \left(p \circ \kappa_2\right) : \
A_1 *_D A_2 \to \mathcal A/C_0\left([1,\infty),B\right)
$$
is defined. By composing this $*$-homomorphism with a continuous
right-inverse for $p$, whose existence follows from the Bartle-Graves
selection theorem, we get an asymptotic homomorphism $\Phi : A_1 *_D A_2 \to M(B)$ such that $q_B
\circ \Phi_t = \varphi \oplus \left(\psi_1 *_D
      \psi_2\right)$ for all $t$.
\end{proof}

\begin{cor}\label{nuclearcor} Let $A_1,A_2$ and $B$ be separable
  $C^*$-algebras, $B$ stable. Let $D$ be a common $C^*$-subalgebra of
  $A_1$ and $A_2$, i.e. $D \subseteq A_1$ and $D \subseteq
  A_2$. Assume that
\begin{enumerate}
\item[1)] $D$ is nuclear, and
\item[2)] that all extensions of $A_1$ by $B$ and all extensions of $A_2$
  by $B$ are semi-invertible.
\end{enumerate}

It follows that all extensions of $A_1 *_D A_2$ by $B$ are semi-invertible.
\end{cor}
\begin{proof} It is well-known that condition 2) of Theorem
  \ref{amalsemi3} is fullfilled when $D$ is nuclear. That condition 1)
  also holds follows from Lemma 2.2 of \cite{Th2}.
\end{proof}

One important virtue of Theorem \ref{amalsemi3} and Corollary \ref{nuclearcor} when compared with
Theorem \ref{amalthm} is the improved symmetry between assumptions and
conclusions which allows to use it iteratively, for example to reach the
following conclusion: Let $A_1,A_2,A_3,A_4$ be separable $C^*$-algebras,
$D_1 \subseteq A_1, D_1 \subseteq A_2$, and $D_2 \subseteq A_3, D_2
\subseteq A_4$ common $C^*$-algebras. Assume that the $A_i$'s and
$D_i$'s are all nuclear, and let $E$ be a common nuclear
$C^*$-subalgebra of $A_1 *_{D_1} A_2$ and $A_3 *_{D_2} A_4$. It follows
that all extensions of
$$
\left(A_1 *_{D_1} A_2 \right) *_E \left(A_3 *_{D_2} A_4\right)
$$
by a separable stable $C^*$-algebra $B$ are semi-invertible.

\section{Full group $C^*$-algebras}

In this section we collect some consequences of Theorem \ref{amalthm}
and Theorem \ref{amalsemi3} for the semi-invertibility of extensions by
full group $C^*$-algebras.

\begin{prop}\label{1} Let $G_1,G_2$ be countable discrete groups and $H
  \subseteq G_i, i = 1,2$, a common subgroup. Set $G = G_1 *_H G_2$
  and let $B$ be a separable stable $C^*$-algebra. Assume that
  $\Ext(C^*(G_i),B), i = 1,2$, are both groups. It follows that every
  extension of $C^*(G)$ by $B$ is strongly homotopy invertible.
\end{prop}
\begin{proof} We can apply Theorem \ref{amalthm} because $C^*(G) =
  C^*(G_1) *_{C^*(H)} C^*(G_2)$. Indeed, there are canonical conditional
  expectations $C^*(G) \to C^*(H)$ and $C^*(G) \to C^*(G_i), i = 1,2$,
  so any absorbing
$*$-homomorpshism $\alpha_0 : C^*(G) \to M(B)$, whose existence is
guaranteed by
\cite{Th1}, will meet the requirements in 1) of Theorem
\ref{amalthm} by Lemma 2.1 of \cite{Th2}. The conclusion of the corollary follows therefore from
Theorem \ref{amalthm}.

\end{proof}

Similarly, Theorem \ref{amalsemi3} implies the following

\begin{prop}\label{groupcor} Let $G_i, i = 1,2$, be discrete
  countable groups with a common subgroup $H \subseteq G_i, i
  = 1,2$, and $B$ a separable stable $C^*$-algebra.  Let $G_1 *_H G_2$
  be the amalgamated free product group and let $B$ be a separable
  stable $C^*$-algebra. Assume that
\begin{enumerate}
\item[1)] $\Ext(C^*(H),B)$ and $\Ext\left(C^*(H), C_0[1,\infty)
    \otimes B\right)$ are both group, and
\item[2)] for both $i = 1$ and $i = 2$ every extension of $C^*(G_i)$
  by $B$ is semi-invertible.
\end{enumerate}

It follows that every extension of $C^*(G_1 *_H G_2)$
  by $B$ is semi-invertible.
\end{prop}

As is wellknown, condition 1) in Proposition \ref{groupcor} is satisfied
when $H$ is amenable, but it is also satisfied for certain
non-amenable groups, e.g. free groups or
an
amalgamated free product of amenable groups over a finite subgroup.

We shall finish this paper by showing that the conclusions of
Propositions \ref{1} and \ref{groupcor}, and partly also the
conclusion of Theorem \ref{thm1}, are preserved by taking the product
of the group with a group of the form $\mathbb Z^k \oplus H$, with $H$ finite.

\begin{lemma}\label{stab} Let $A$ and $B$ be separable $C^*$-algebras,
  $B$ stable. There are semi-group homomorphisms $\mu : \Ext(A,B) \to
  \Ext(A \otimes \mathbb K,B)$ and $\nu : \Ext(A\otimes \mathbb K,B) \to
  \Ext(A ,B)$ such that $\mu \circ \nu(x) \oplus 0 = x \oplus 0$ for
  all $x \in \Ext(A\otimes \mathbb K,B)$ and $\nu \circ \mu(y) \oplus
  0 = y \oplus 0$ for all $\Ext(A,B)$.
\begin{proof} Since $B$ is stable we can identify $B$ and $\mathbb K
  \otimes B$. Let $e$ be a minimal projection in $\mathbb K$ and let
  $V \in M(\mathbb K \otimes \mathbb K \otimes B)$ be an isometry such
  that $VV^* = e \otimes 1_{\mathbb K \otimes B}$. Then $\alpha(x) = V^*(e
  \otimes x)V$ is an isomorphism $\alpha : \mathbb K \otimes B \to
  \mathbb K \otimes \mathbb K \otimes B$, giving us isomorphisms
  $M\left(\mathbb K \otimes B\right) \to M\left( \mathbb K \otimes
    \mathbb K \otimes B\right)$ and $Q\left(\mathbb K \otimes B\right) \to Q\left( \mathbb K \otimes
    \mathbb K \otimes B\right)$ which we also denote by $\alpha$. Let
  $s : A \to \mathbb K \otimes A$ be the $*$-homomorphism $s(a) = e
  \otimes a$. We can then define a map
\begin{equation}\label{map1}
\Ext\left( \mathbb K \otimes A, \mathbb K \otimes \mathbb K \otimes
  B\right) \to \Ext\left(A,\mathbb K \otimes B\right)
\end{equation}
by $\varphi \mapsto \alpha^{-1} \circ \varphi \otimes s$. To get a map in the other direction note that the canonical embedding
$\mathbb K \otimes M(\mathbb K \otimes B) \subseteq M(\mathbb K
\otimes \mathbb K \otimes B)$ induce a $*$-homomorphism $L : \mathbb K
\otimes Q(\mathbb K \otimes B) \to Q(\mathbb K \otimes \mathbb K
\otimes B)$ which we can use to define a map
\begin{equation}\label{map2}
\Ext\left( A, \mathbb K  \otimes
  B\right) \to \Ext\left(\mathbb K \otimes A, \mathbb K \otimes
  \mathbb K \otimes B\right)
\end{equation}
by $\varphi \mapsto L \circ \left(\id_{\mathbb K} \otimes
  \varphi\right)$. Then $\alpha^{-1} \circ \left(L \circ \left(\id_{\mathbb K} \otimes
    \varphi\right)\right) \circ s =  \Ad q_{\mathbb K \otimes B}(W) \circ \varphi$ for some
isometry $W \in M(\mathbb K \otimes B)$, showing that
$$
\left[ \left(\alpha^{-1} \circ \left(L \circ \left(\id_{\mathbb K} \otimes
    \varphi\right)\right) \circ s\right) \oplus 0 \right] =
\left[\varphi \oplus 0 \right]
$$
in $\Ext\left( A, \mathbb K  \otimes
  B\right)$.

Consider next an extension $\varphi : \mathbb K \otimes A \to Q\left(\mathbb K \otimes
  \mathbb K \otimes B\right)$. Note that
$$
L \circ \left(\id_{\mathbb K} \otimes \left(\alpha^{-1} \circ \varphi
   \circ s\right)\right)(k \otimes a) = L \left( k \otimes \alpha^{-1}\left(\varphi(e \otimes a)\right)\right)
$$
on simple tensors, $k \in \mathbb K, a \in A$. Since the automorphism
of $Q(\mathbb K  \otimes  \mathbb K \otimes A)$ which interchange the
two copies of $\mathbb K$ is given by a unitary in $M\left(\mathbb K
  \otimes \mathbb K \otimes B\right)$, the extension $L \circ \left(\id_{\mathbb K} \otimes \left(\alpha^{-1} \circ \varphi
   \circ s\right)\right)$ is unitarily equivalent to an extension
$\psi : \mathbb K \otimes A \to Q(\mathbb K \otimes \mathbb K \otimes
B)$ such that
$$
\psi(k \otimes a) =  L \left( e \otimes \alpha^{-1}\left(\varphi(k \otimes a)\right)\right)
$$
on simple tensors. Since $L \left( e \otimes \alpha^{-1}\left(\varphi(k \otimes a)\right)\right)
= \Ad q_{\mathbb K \otimes \mathbb K \otimes B}(V)\left(\varphi(k
  \otimes a)\right)$, we see that the two maps, (\ref{map1}) and
(\ref{map2}) are inverses of each other, up to addition by $0$. Since
both maps clearly are semi-group homomorphisms, the proof is complete.

\end{proof}
\end{lemma}

\begin{cor}\label{matrix!} Let $A$ and $B$ be separable
  $C^*$-algebras, $B$ stable. Then all extensions of $A$ by $B$ are
  semi-invertible or strongly homotopy invertible if and only if the
  same is true for all extensions of $M_n(A)$ by $B$, for any $n \in
  \mathbb N$.
\end{cor}

\begin{lemma}\label{sum} Let $A_1,A_2$ and $B$ be separable
  $C^*$-algebras, $B$ stable. Assume that all extensions of $A_i$ by
  $B$ are semi-invertible or are strongly homotopy invertible (with an invertible inverse), $i =1,2$. It follows that all extensions of
  $A_1 \oplus A_2$ by $B$ have the same property.
\begin{proof} Let $p_i : A_1 \oplus A_2 \to A_i \subseteq A_1 \oplus A_2, i =1,2$, be the
  canonical projections, and consider an extension $\varphi : A_1
  \oplus A_2 \to Q(B)$. By a standard rotation argument $\varphi \oplus
  0$ is
  strongly homotopic to the sum $\left(\varphi \circ p_1\right) \oplus
  \left(\varphi \circ p_2\right)$. The conclusion follows from this by
  use of Theorem \ref{basic}.
\end{proof}
\end{lemma}

By combining Corollary \ref{matrix!} and Lemma \ref{sum} we get the following.

\begin{cor}\label{finte} Let $A,F$ and $B$ be separable
  $C^*$-algebras, $B$ stable, $F$ finite dimensional. Assume that all extensions of $A$ by
  $B$ are semi-invertible or are strongly homotopy invertible (with an
  invertible inverse). It follows that all extensions of $F \otimes A$
  by $B$ have the same property.
\end{cor}

In particular, it follows that if $G$ is a countable discrete group
with the property that all extensions of $C_r^*(G)$ by $B$ are
semi-invertible or
strongly homotopy invertible (with an invertible inverse), then the same
is true for $C^*_r\left(H\times G\right)$ for any finite group $H$.

\begin{lemma}\label{anne!!} Let $A$ and $B$ be separable
  $C^*$-algebras, $B$ stable. Assume that all extensions of $A$ by $B$
  are semi-invertible or strongly homotopy invertible. It follows that
  all extensions of $C(\mathbb
  T)\otimes A$ by $B$ have the same property.
\begin{proof} Let $\chi $ be the automorphism of $C(\mathbb T) \otimes
  A$ such that $\chi(f)(z) = f(\overline{z})$ and let $\ev : C(\mathbb T)
  \otimes A \to A$ be evaluation at $1 \in \mathbb T$. As is wellknown
  the $*$-homomorphism $C(\mathbb T) \otimes A \to M_2\left(C(\mathbb
    T) \otimes A \right)$ defined such that
$$
f \mapsto \left( \begin{smallmatrix} f & \\ & \chi(f)
  \end{smallmatrix} \right)
$$
is homotopic to a $*$-homomorphism which factorizes through $\ev$. It
follows that for any extension $\varphi : C(\mathbb T) \otimes A \to Q(B)$
the extension $\varphi \oplus \varphi \circ \chi$ is strongly
homotopic to an extension of the form $\psi \circ \ev$, where $\psi :
A \to Q(B)$ is an extension of $A$ by $B$. By assumption there is an extension $\psi'$ of $A$ by $B$
such that $\psi \oplus \psi'$ is either asymptotically split or strongly
homotopic to a split extension. It follows that
$\varphi
\oplus \varphi \circ \chi \oplus \psi' \circ \ev $ has the same
property by Theorem \ref{basic}. Hence $\varphi$ is
semi-invertible or strongly homotopy invertible, as the case may be.
\end{proof}
\end{lemma}

\begin{prop}\label{directsum} Let $G$ be a countable discrete
  group, $H$ a finite group and $k \in \mathbb N$. Let $B$ be a
  separable stable $C^*$-algebra and assume that all extensions of
  $C^*_r(G)$, (resp. $C^*(G)$), by $B$ are semi-invertible or strongly
  homotopy invertible. It follows
  that all extensions of $C^*_r\left( \mathbb Z^k \times H \times
    G\right)$, (resp. $C^*\left( \mathbb Z^k \times H \times
    G\right)$), by $B$ have the same property.
\begin{proof} Note that $C^*_r\left(\mathbb Z^k \times H \times
    G\right) \simeq C\left(\mathbb T^k\right) \otimes C^*(H) \otimes
  C_r^*(G)$, and that $C^*(H)$ is finite dimensional. It follows then
  from Corollary \ref{finte} and Lemma \ref{anne!!} that all extensions of $C^*_r\left( \mathbb Z^k \times H \times
    G\right)$ by $B$ are semi-invertible or strongly homotopy
  invertible if $C^*_r(G)$ has this property. The same argument works for
  the full group $C^*$-algebra.
\end{proof}
\end{prop}

Finally, we observe that it is also possible to use Theorem \ref{amalthm}
and Theorem \ref{amalsemi3} to prove semi-invertibility for extensions
of the full group $C^*$-algebra of certain HNN-extensions by using the
realization obtained by Ueda in \cite{U} of such group $C^*$-algebras  as amalgamated free products.


\begin{thebibliography}{WWW} 



\bibitem[A]{A}
{J. Anderson},
{\it A $C^*$-algebra for which $\Ext(A)$ is not a group},
Ann. of Math. {\bf 107} (1978), 455--458.

\bibitem[BDF1]{BDF1}
{L.G. Brown, R.G. Douglas and P.A. Fillmore},
{\it Unitary equivalence modulo the compact operators and
extensions of $C^*$-algebras}, Proc. Conf. on Operator Theory,
Lecture Notes in Mathematics {\bf 345}, Springer Verlag (1973), 58--128.



\bibitem[BDF1]{BDF2}
{\bysame},
{\it Extensions of $C^*$-algebras and $K$-theory}, Ann. of Math.
{\bf 105} (1977), 265--324.





\bibitem[Br]{Br} L. Brown, {\em Ext of certain free product
    $C^*$-algebras}, J. Operator Th. {\bf 6} (1981), 135-141.


\bibitem[CCJJV]{CCJJV}
P-A. Cherix, M. Cowling, P. Jolissaint, P. Julg and A. Valette, \emph{Groups with the Haagerup Property}, Birkh\"auser Verlag, (2001).

\bibitem[CH]{CH} A. Connes and N. Higson, {\em D\'eformations,
    morphismes asymptotiques et K-th\'eories bivariante},
  C.R. Acad. Sci. Paris S\'er I Math. {\bf 311} (1990), 101-106.



\bibitem[C]{C} J. Cuntz, {\em K-theoretic amenability for discrete
    groups}, J. Reine u. Angew. Math. {\bf 344} (1983), 180-195.



\bibitem[DE]{DE} M. Dadarlat and S. Eilers, {\em Asymptotic Unitary
    Equivalence in KK-theory}, K-theory {\bf 23} (2001), 305-322.







\bibitem[HT]{HT}
U. Haagerup and S. Thorbj\o rnsen, \emph{A new application of random matrices: $ \Ext(C^*_{red}(F_2)) $ is not a group},
Ann. of Math. {\bf 162} (2005), 711-775.

\bibitem[HS]{HS} D. Hadwin and J. Shen, {\em Some examples of
    Blackadar and Kirchberg's MF algebras}, Preprint, arXiv:0806.4712.

\bibitem[HLSW]{HLSW} D. Hadwin, J. Li, J. Shen, J. Wang, \emph{Reduced
    free products of unital AH algebras and Blackadar and Kirchberg's
    MF algebras}, Preprint, arXiv:0812.0189v1






\bibitem[HK]{HK} N. Higson and G. Kasparov, {\em E-theory and
    KK-theory for groups which act properly and isometrically on
    Hilbert space}, Invent. Math. {\bf 144} (2001), 23-74.


\bibitem[K]{K} G. Kasparov, {\em Equivariant KK-theory and the Novikov
    conjecture}, Invent. Math. {\bf 91} (1988), 513-572.



\bibitem[L]{L} H. Lin, {\em Stable approximate unitary equivalence of
    homomorphisms}, J. Oper. Theory {\bf 47} (2002), 343-378.




\bibitem[M]{M} V. Manuilov, {\em Asymptotic representations of the
    reduced $C^*$-algebra of a free group: an example}, Bull. London
  Math. Soc. {\bf 40} (2008), 838-844.

\bibitem[MT1]{MT1} V. Manuilov and K. Thomsen, {\em E-theory is a special case of
    KK-theory}, Proc. London Math. Soc. {\bf 88} (2004), 455-478.


\bibitem[MT2]{MT2} \bysame,  {\em Relative K-homology
    and normal operators}, J. Operator Th. {\bf 62} (2009), 249-279.





\bibitem[MT3]{MT3} \bysame, {\em The Connes-Higson
    construction is an isomorphism}, J. Func. Anal. {\bf 213} (2004), 154-175.



\bibitem[MT4]{MT4} \bysame, {\em On the lack of inverses to
    $C^*$-extensions related to property T groups},
  Can. Math. Bull. {\bf 50} (2007), 268-283.












\bibitem[P]{P} G. K. Pedersen, {\em $C^*$-algebras and their
    automorphisms group}, Academic Press, New York, 1979.


\bibitem[Se]{Se} J.A. Seebach, {\em On the Reduced Amalgamated Free
    Products of $C^*$-algebras and the MF-Property}, arXiv:1004.3721



\bibitem[S]{S} J.-P. Serre, {\em Trees}, Springer Verlag, Berlin, 1977.






\bibitem[ST]{ST} J. A. Seebach and K. Thomsen, {\em Extensions
    of the reduced group $C^*$-algebra of a free product of amenable
    groups}, Adv. Math. {\bf 223} (2010), 1845-1854.







\bibitem[Th1]{Th1}  K. Thomsen,, {\em On absorbing extensions},
  Proc. Amer. Math. Soc. {\bf 129} (2001), 1409-1417.




\bibitem[Th2]{Th2}  \bysame, {\em On the $KK$-theory and the $E$-theory of
  amalgamated free products of $C\sp \ast$-algebras},
J. Funct. Anal. {\bf 201} (2003), 30-56.



\bibitem[Th3]{Th3} \bysame, {\em Homotopy invariance in E-theory},
Homology, homotopy and applications {\bf 8} (2006), 29-49.



\bibitem[Th4]{Th4} \bysame, {\em All extensions of $C^*_r(\mathbb F_n)$ are semi-invertible}, Math. Ann. {\bf 342} (2008), 273-277.




\bibitem[Tu]{Tu} J.L. Tu, {\em La conjecture de Baum-Connes pour les
    feuilletages moyennables}, K-theory {\bf 17} (1999), 215-264.

\bibitem[U]{U} Y. Ueda, {\em Remarks on HNN extensions in operator algebras},
  Illinois J. Math. {\bf 52} (2008), 705-725. 








\end{thebibliography}
 \end{document}